\documentclass[12pt]{article}
\usepackage[pdftex]{graphicx}
\usepackage{amsmath}
\usepackage{amssymb}
\usepackage{amsthm} 

\newtheorem{thm}{Theorem}[section] 

\newtheorem{con}[thm]{Conjecture}

\newtheorem{lem}[thm]{Lemma}

\newtheorem{prop}[thm]{Proposition}

\title{The Maximum of the Maximum Rectilinear Crossing Numbers of $d$-regular Graphs of Order $n$}

\author{Matthew Alpert\footnote{Lawrence High School, Cedarhurst, NY, USA. \textit{E-mail address: }\texttt{mna851@aol.com}}, Elie Feder\footnote{Department of Mathematics and Computer Science, Kingsborough Community College-CUNY, Brooklyn, NY, USA. \textit{E-mail address: }\texttt{efeder@kbcc.cuny.edu}}, and Heiko Harborth\footnote{Diskrete Mathematik, Technische Universitaet, Braunschweig, Germany. \textit{E-mail address: }\texttt{H.Harborth@tu-bs.de}}}

\begin{document}
\maketitle

\begin{abstract}
We extend known results regarding the maximum rectilinear crossing number of the {\it cycle graph} ($C_n$) and the {\it complete graph} ($K_n$) to the class of general {\it d-regular graphs} $R_{n,d}$. We present the {\it generalized star} drawings of the $d$-regular graphs  $S_{n,d}$ of order $n$ where $n+d\equiv 1 \pmod 2 $ and prove that they maximize the maximum rectilinear crossing numbers. A {\it star-like} drawing of $S_{n,d}$ for $n \equiv d \equiv 0 \pmod 2$ is introduced and we conjecture that this drawing maximizes the maximum rectilinear crossing numbers, too. We offer a simpler proof of two results initially proved by Furry and Kleitman \cite{cycle1} as partial results in the direction of this conjecture.
\end{abstract}

\section{Introduction}
Let $G$ be an abstract graph with vertex set $V(G)$ and edge set
$E(G)\subset V(G) \times V(G)$. The \textit{order} of a graph $G$ is
defined as the cardinality of $V(G)$.  A \textit{drawing} of the graph $G$ is a representation of $G$ in
the plane such that the elements of $V(G)$ correspond to points in the plane,
and the elements of $E(G)$ correspond to continuous arcs connecting two vertices and having at most one point in common, either a vertexpoint or a crossing. A
\textit{rectilinear drawing} is a drawing of a graph in which all
edges are represented as straight line segments in the plane.

The \textit{degree} of a vertex $v \in V(G)$ is defined as the number of edges in $E(G)$ containing $v$
as an endpoint.  If all vertices of a graph have the same degree,
then the graph is called \textit{regular}. Specifically, if all the
vertices have degree $d$, the graph is called $d$\textit{-regular}.
The \textit{cycle} $C_n$ is a connected $2$-regular graph. The
\textit{complete graph} $K_n$ is a graph on $n$ vertices, in which
any two vertices are connected by an edge, or equivalently an
$(n-1)$-regular graph. The class of $d$-regular graphs of order $n$
will be denoted $R_{n,d}$.

In a drawing of a graph, a \textit{crossing} is defined to be the
intersection of exactly two edges not at a vertex.  The \textit{crossing number} of an abstract graph, $G$,
 denoted ${\rm cr}(G)$, is defined as the minimum number of edge crossings over all nonisomorphic
drawings of $G$.  The \textit{minimum rectilinear crossing
number} of a graph $G$, denoted $\overline{\rm cr}(G)$, is defined as
the minimum number of edge crossings over all nonisomorphic
rectilinear drawings of $G$.

Analogously, the \textit{maximum crossing number}, denoted by ${\rm CR}(G)$, is
defined as the maximum value of edge crossings over all nonisomorphic
drawings of $G$.  The \textit{maximum rectilinear crossing
number} of a graph $G$, denoted by $\overline{\rm CR}(G)$, is
defined to be the maximum number of edge crossings over all
nonisomorphic rectilinear drawings of $G$.  Throughout this paper we will also define $\overline{\rm CR}(R_{n,d})$ to be the maximum of the maximum rectilinear crossing numbers throughout the class of graphs.

The maximum crossing number and maximum rectilinear crossing number
have been studied for several classes of graphs (see \cite{maxrect1}, \cite{max1}, \cite{max2}, \cite{max3}, \cite{max4}).  Most relevant
to this paper are studies of the maximum rectilinear crossing number
of $C_n$ (a $2$-regular graph) and of $K_n$ ($(n-1)$-regular graph).
In \cite{complete1} it is shown that
$$\overline{\rm CR}(K_n)=\overline{\rm CR}(R_{n,n-1})=\binom{n}{4}.$$
In \cite{cycle1}, \cite{cycle2} it is proved that
$$ \overline{\rm CR}(C_n)= \begin{cases}
\frac{1}{2}n(n-3) & \text{if $n$ is odd,} \\
\frac{1}{2}n(n-4) + 1 & \text{if $n$ is even.}
\end{cases} $$

This paper makes a natural generalization from these two results.
Namely, it finds an expression for the maximum $\overline{\rm CR}(R_{n,d})$ of all maximum rectilinear crossing
numbers for the class $R_{n,d}$ of all $d$-regular graphs of order $n$, where $2 \leq d \leq
n-1$.  We present a {\it star-like} drawing of a $d$-regular graph $S_{n,d}$ for $n$ and $d$ of different parity and prove that it maximizes the maximum rectilinear crossing numbers.  A {\it star-like} drawing of the $d$-regular graph $S_{n,d}$ for even $n$ and $d$ is introduced and we conjecture that this drawing maximizes the maximum rectilinear crossing numbers offering proofs for $d=2$ and $d=n-2$ as partial results in the direction of this conjecture.  

We present here an interesting method of generalizing the maximum
rectilinear crossing number of $C_n$ and $K_n$ to the more general class $R_{n,d}$ of $d$-regular graphs of order $n$.  Finding the minimum rectilinear crossing number of the complete graph, $K_n$, is a well-known and widely-investigated open problem in computational geometry.
For $n<17$, $\overline{\rm cr}(K_n)$ is known, and for $n \geq 18$ only bounds are known (see \cite{min1}, \cite{min2}, \cite{min3}). Perhaps future research can investigate  $\overline{\rm cr}(R_{n,d})$ where $d<n-1$ as a tool to gain insight into the minimum rectilinear crossing number of $K_n$.

\medskip

In Sections 2.1 and 2.2 we outline the construction of the {\it generalized star-like} drawings of $S_{n,d}$ and present a lower bound for $\overline{\rm CR}(R_{n,d})$.  In Section 3.1 we present an upper bound for $\overline{\rm CR}(R_{n,d})$, where $n+d\equiv 1 \pmod 2$, and note that the {\it star-like} drawing of $S_{n,d}$ attains this maximum.  In Section 3.2 we conjecture the upper bound of $\overline{\rm CR}(R_{n,d})$ where $n \equiv d \equiv 0 \pmod 2$ and offer a partial result in the direction of this conjecture by proving its validity for the case $d=2$.  In Section 3.3 we offer simpler proofs of the maximum crossing number of $C_n$ and of $\overline{\rm CR}(R_{n,2})$ where $n \equiv 0 \pmod 2$ than those of Furry and Kleitman \cite{cycle1} and in Section 3.4 we remark on this paper's generalization of previous results.  Section 3.5 contains some computational results regarding  $\overline{\rm CR}(R_{n,d})$.

\section{Lower Bounds of $\overline{\rm CR}(R_{n,d})$}

We first note that there is no $d$-regular graph of order $n$ where $n$ and $d$ are both odd since the number $nd$ of endvertices cannot twice count the number of edges.  Thus, we will only consider the two cases $n+d \equiv 1 \pmod 2$ and $n \equiv d \equiv 0 \pmod 2$.  

\subsection{Lower bound of $\overline{CR}(R_{n,d})$ where $n+d \equiv 1 \pmod 2$} 

The number of crossings in a special rectilinear {\it star-like} drawing implies the following lower bound. 

\begin{prop}
$$ \overline{CR}(R_{n,d}) \geq \frac{nd}{24}(3nd-2d^2-6d+2)$$ if $n+d \equiv 1 \pmod 2$.
\end{prop}

\begin{proof}
Consider a rectilinear drawing of $K_n$ where the vertices are arranged as those of a convex $n$-gon.  Step by step we delete all diagonals of lengths $1,2, \ldots, k-1$.  We proceed by counting the number of crossings we remove from the drawing by now deleting the $n$ diagonals of length $k$. There are $k-1$ vertices in one of the halfplanes each of the diagonals of length $k$ divides the drawing into.  Each of these vertices will have $n-1-(2(k-1))=n-2k+1$ edges emanating from it which intersect the original diagonal of length $k$.  However, each diagonal of length $k$ intersects $2(k-1)$ other diagonals of length $k$.  Since these crossings are counted twice, we find that there are $n(k-1)$ crossings between diagonals of length $k$.  These are also counted twice in the sum $n(k-1)(n-2k+1)$ and thus we only remove $n(k-1)(n-2k+1)-n(k-1)=n(k-1)(n-2k)$ crossings in deleting all diagonals of length $k$ provided all shorter diagonals have been previously deleted.  Therefore we obtain
        
$$\overline{CR}(R_{n,d}) \geq \binom{n}{4} - \displaystyle\sum_{i=1}^{k-1} n(i-1)(n-2i)$$
$$=\binom{n}{4} - \frac{1}{6}n(k-1)(k-2)(3n-4k).$$

After these deletions there are $d=n-1-2(k-1)$ edges emanating from each vertex.  Substituting $k=\frac{1}{2}(n-d+1)$ into the closed form of the sum above we obtain the desired result.  We call this drawing of $K_n$ without the diagonals of lengths $1$ through $k-1=\frac{1}{2}(n-d-1)$ the \textit{generalized star} drawing of $S_{n,d}$ in $R_{n,d}$ (see Figure 1).

\begin{figure}[ht]
\begin{center}
\includegraphics[scale=.20]{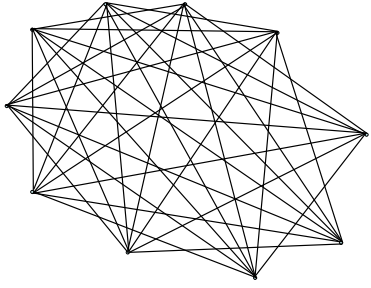}
\caption{The generalized star drawing of $S_{10,7}$ in $R_{10,7}$.}
\end{center}
\end{figure}

\end{proof}

\subsection{Lower bound of $\overline{CR}(R_{n,d})$ where $n \equiv d \equiv 0 \pmod 2$}

The number of crossings in the special rectilinear {\it star-like} drawing implies the following lower bound for $\overline{CR}(R_{n,d})$ where $n \equiv d \equiv 0 \pmod 2$.

\begin{prop}

$$ \overline{\rm CR}(R_{n,d}) \geq \begin{cases}
\frac{1}{24}nd(3nd-2d^2-6d-1) & \text{if $n \equiv d \equiv n/(n,k) \equiv 0 \pmod 2, $ }\\ \\
\frac{1}{24}nd(3nd-2d^2-6d-1)\\ - \frac{1}{4}(n,k)(2d-3) & \text{if $n \equiv d \equiv 0 \pmod 2$} \\ & \text{and   $n/(n,k)\equiv 1 \pmod 2 $} 
\end{cases} $$

\text{where $k=\frac{1}{2}(n-d)$.}

\end{prop}

\begin{proof}

For $n \equiv d \equiv 0 \pmod 2$ we use the generalized star drawing of $S_{n,d+1}$ for $d+1 = n-2k+1$ and delete one edge at each vertex to obtain a star-like drawing of $S_{n,d}$ with $d=n-2k$ (an even number).  The diagonals of length $k$ in $S_{n,d+1}$ determine $(n,k)$ cycles each of order $n/(n,k)$.

If $n/(n,k) \equiv 0 \pmod 2$ we can delete every second edge of every cycle (see Figure 2).  In removing these edges we remove $$\frac{1}{2}n(k-1)(n-2k+1) - \frac{1}{4}n(k-1)=\frac{1}{2}n(k-1)(n-2k+\frac{1}{2})$$ edge crossings from the drawing.  Subtracting this from the bound in Proposition 2.1 it follows that
$$\overline{CR}(R_{n,d}) \geq \binom{n}{4} - \frac{1}{6}n(k-1)(k-2)(3n-4k)- \frac{1}{2}n(k-1)(n-2k+\frac{1}{2}).$$
Substituting $k=\frac{1}{2}(n-d)$ gives the desired result. 

\begin{figure}[h]
\begin{center}
\includegraphics[scale=.20]{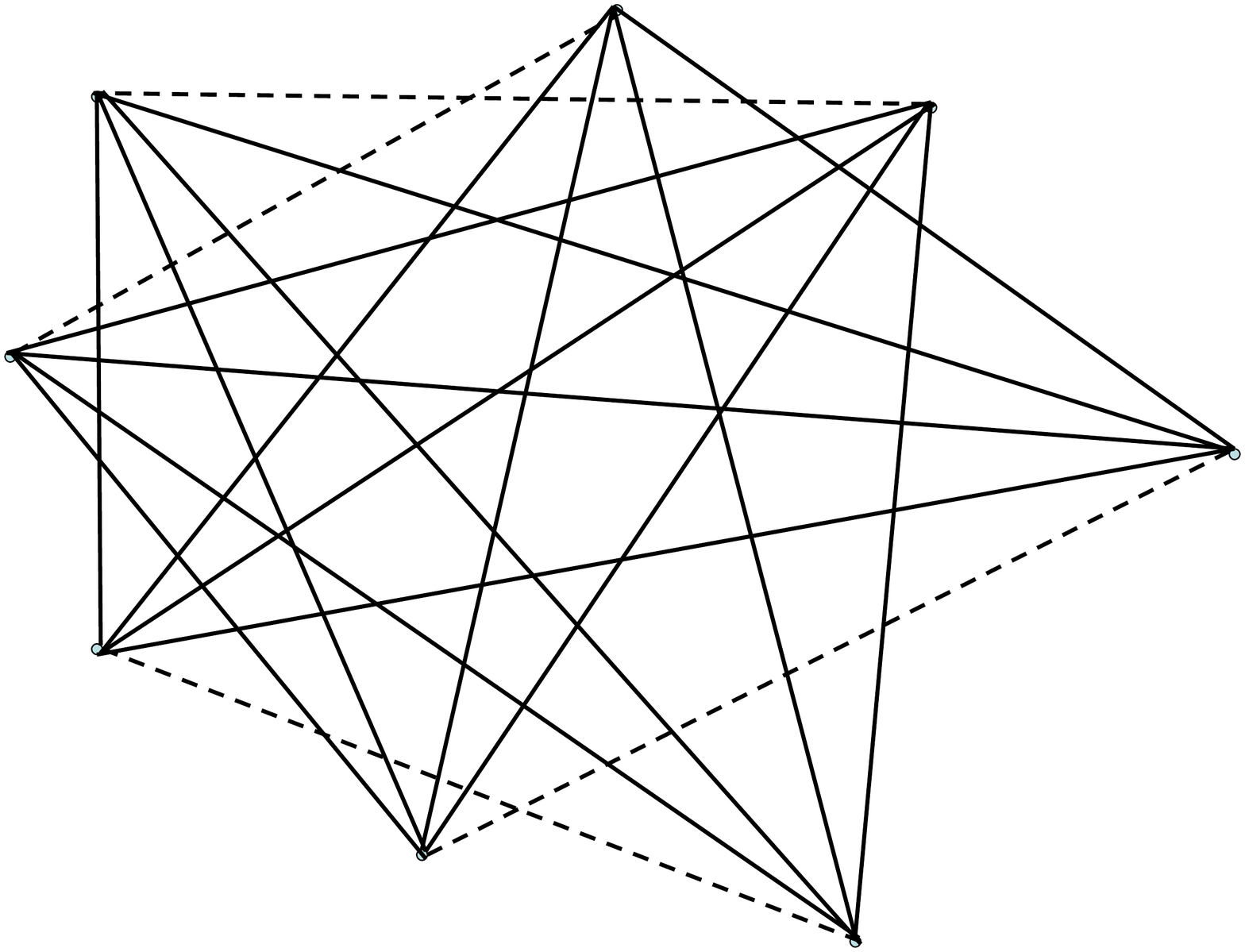}
\includegraphics[scale=.20]{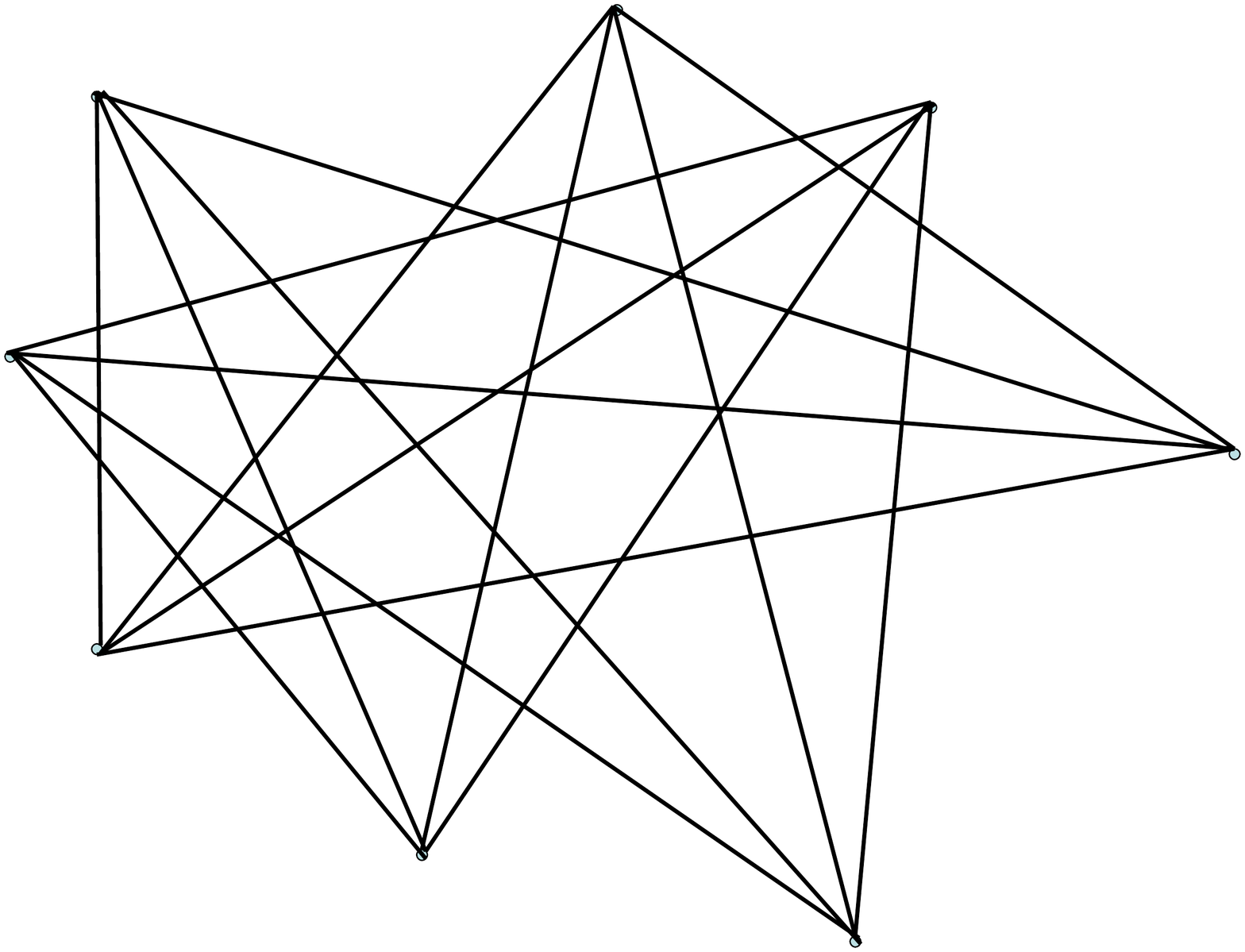}
\caption{In the drawing on the left, every second diagonal in the cycles of order $k=4$ are dashed.   In the drawing on the right those edges are removed yielding a star-like drawing of $S_{8,4}$ in $R_{8,4}$.}
\end{center}
\end{figure}

If $n/(n,k) \equiv 1 \pmod 2$ then the diagonals of length $k$ determine $(n,k) \equiv 0 \pmod 2$ cycles of odd order.  We partition these cycles into $\frac{1}{2}(n,k)$ pairs.  For each pair we delete a diagonal of length $k+1$ connecting two vertices of these cycles.  For each of these diagonals of length $k+1$ we keep the diagonals of length $k$ which emanate from their endpoints and then delete their neighbor edges and every second of the remaining edges within the cycles of order $n/(n,k)$ (see Figure 3).  Thus we remove $\frac{1}{2}(n,k)$ edges of length $k+1$ and $$\frac{1}{2}({n/(n,k)-1})(n,k)=\frac{1}{2}(n-(n,k))$$ diagonals of length $k$.  In removing these edges we remove
$$\frac{1}{2}(n-(n,k))(k-1)(n-2k+1)+\frac{1}{2}(n,k)(k)(n-2k+1)$$ $$ - \frac{1}{2}(n,k)2 - \frac{1}{2}[\frac{1}{2}(n-(n,k))(k-1)+\frac{1}{2}(n,k)k]$$
$$=\frac{1}{2}(n-2k+\frac{1}{2})(kn-n+(n,k))-(n,k)$$    
crossings from the drawing of $S_{n,d+1}$.  It follows that
$$\overline{CR}(R_{n,d}) \geq \binom{n}{4} - \frac{1}{6}n(k-1)(k-2)(3n-4k)$$ $$ -\frac{1}{2}(n-2k+\frac{1}{2})(kn-n+(n,k))-(n,k).$$
Substituting $k=\frac{1}{2}(n-d)$ yields the desired inequality.

\begin{figure}[h]
\begin{center}
\includegraphics[scale=.20]{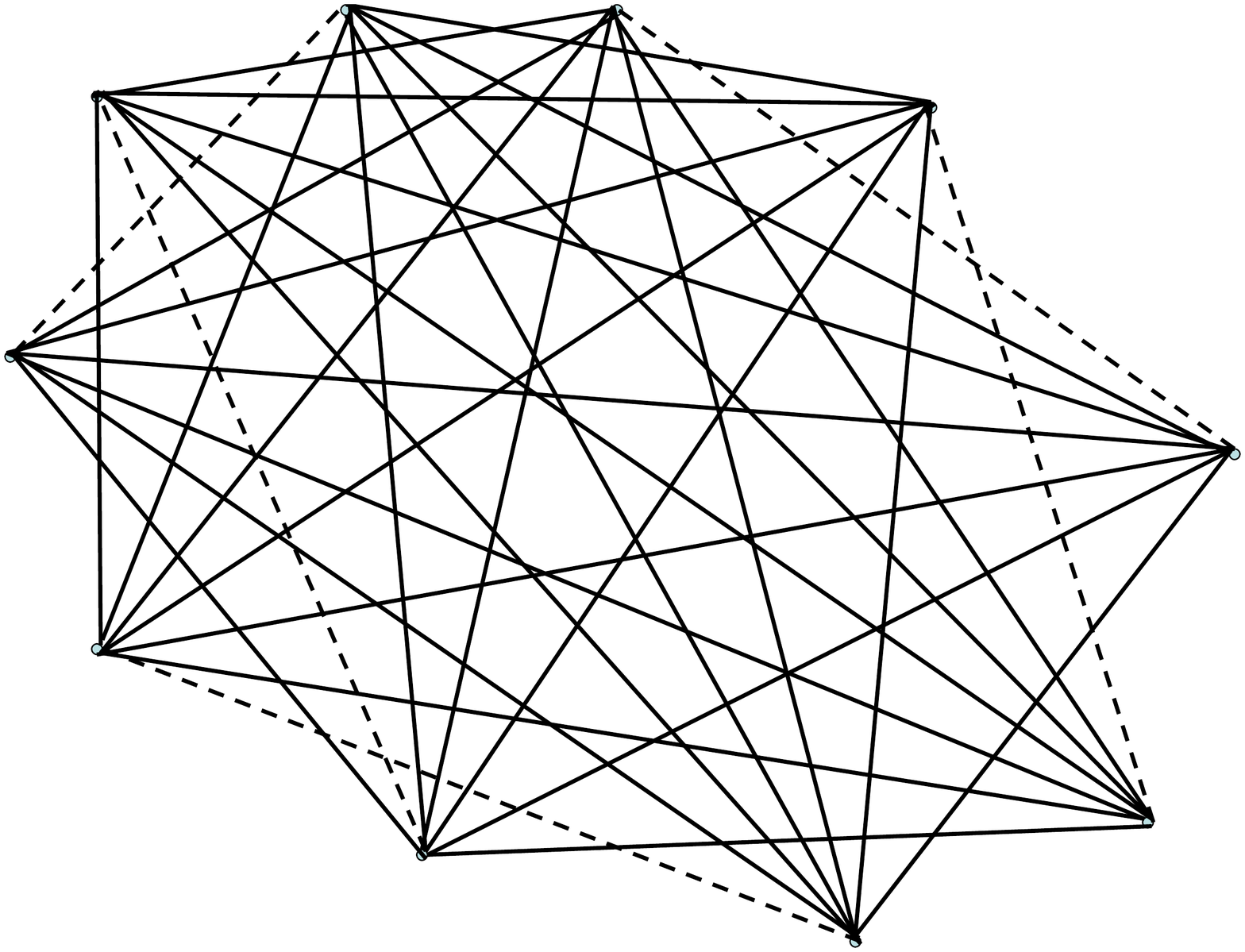}
\includegraphics[scale=.20]{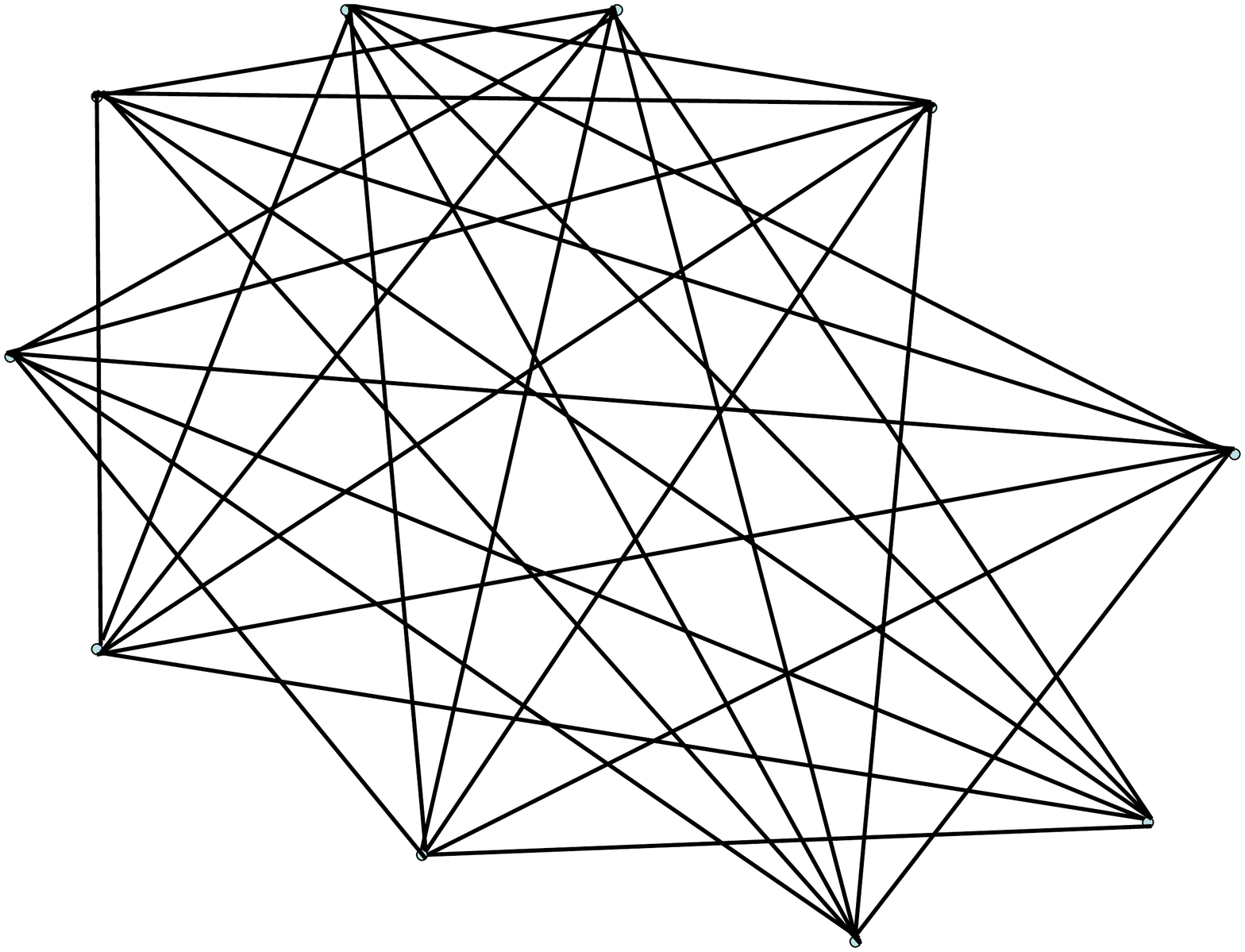}
\caption{In the drawing on the left, $\frac{1}{2}(10,2)=1$ diagonal of length $k+1=3$ has been dashed.  Additionally, every second edge of the cycles $C_5$ emanating from this diagonal's endpoints have been dashed.   In the drawing on the right the dashed edges are removed, yielding a star-like drawing of $S_{10,6}$ in $R_{10,6}$.}
\end{center}
\end{figure}

\end{proof}

\section{Upper Bounds of $\overline{\rm CR}(R_{n,d})$}

In this section we prove that the lower bound obtained in Proposition 2.1 is also an upper bound for $\overline{\rm CR}(R_{n,d})$ where $n+d \equiv 1 \pmod 2$.  In addition we conjecture that the lower bound obtained in Proposition 2.2 is an upper bound and offer a partial result in the direction of this conjecture.

\subsection{Upper bound of $\overline{CR}(R_{n,d})$ where $n+d \equiv 1 \pmod 2$}

The following exact value of $\overline{\rm CR}(R_{n,d})$ will be proved.
\begin{thm} $$ \overline{\rm CR}(R_{n,d}) = \frac{1}{24}nd(3nd-2d^2-6d+2) \ \rm{if}\  n+d \equiv 1 \pmod 2.$$
\end{thm}

\begin{proof}

The lower bound follows from Proposition 2.1, so we proceed by proving that this expression is an upper bound.  Every $d$-regular graph of order $n$ has $\frac{1}{2}nd$ edges.  Every edge can intersect at most $\frac{1}{2}nd-(2d-1)$ other edges.  Thus, a first upper bound is $$ \overline{\rm CR}(R_{n,d}) \leq \frac{1}{2} (\frac{1}{2}nd) (\frac{1}{2}nd-2d+1)=\frac{1}{24}nd(3nd-12d+6). $$ 

Every vertex in a $d$-regular graph is an endvertex for $d$ edges. Let an endvertex be of type $i$ if the edge incident to it divides the drawing of the graph into two halfplanes, one containing $i$ edges emanating from one vertex, and the other containing $d-i-1$ edges emanating from the same vertex (see Figure $4$).  By symmetry we only consider $0 \leq i \leq \lfloor \frac{1}{2}(d-1) \rfloor = D. $
\begin{figure}[ht]
\begin{center}
\includegraphics[scale=.45]{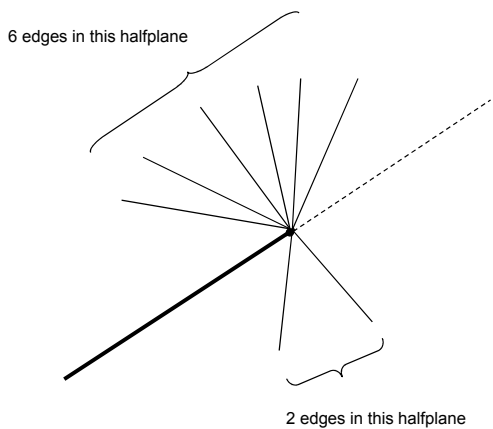}
\caption{The right endvertex of the bold edge is of type $2$ because the \textit{smaller} halfplane determined by this edge contains $2$ edges emanating from this vertex.}
\end{center}
\end{figure}

Let $y_i$ be the number of endvertices of type $i$.  Thus, we have $y_0 + y_1+ \ldots +y_D = dn$.  We call an edge with $i$ edges in a halfplane at one endvertex and $j$ edges in the same halfplane at the other endvertex a type $i,j$ edge.  Let $x_{i,j}$ count the number of type $i,j$ edges (see Figure $5$).   

\begin{figure}[h]
\begin{center}
\includegraphics[scale=.45]{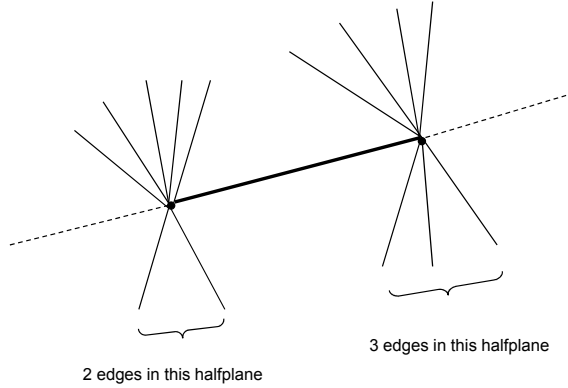}
\caption{The bold edge is a type $2,3$ edge because the left endvertex has $2$ edges emanating from it in the smaller halfplane, and the right endvertex has $3$ edges emanating from it in the same halfplane.}
\end{center}
\end{figure}

Thus, $y_i$ is related to $x_{i,j}$ by the following equation:
\begin{equation}   y_i=2x_{i,i}+\displaystyle\sum_{k=0}^{i-1} x_{k,i} + \displaystyle\sum_{k=i+1}^{D} x_{i,k}.   \end{equation} 

Now, for a type $i,j$ edge, the $i$ edges in the halfplane of one endvertex cannot intersect the $d-j-1$ edges in the opposite halfplane emanating from the other endvertex.  The same holds true for the $j$ edges in the halfplane of one endvertex and the $d-i-1$ edges in the opposite halfplane emanating from the other endvertex.  Therefore, a given type $i,j$ edge determines $i(d-j-1)+j(d-i-1)$ pairs of nonintersecting edges.  A drawing which maximizes the number of edge crossings should minimize the number $M$ of pairs of nonintersecting edges.  Note that it is true that for a given type $i,j$ edge it may be that the $i$ edges from one endvertex and the $j$ edges from the other endvertex will be in different halfplanes.  This will yield $ij + (d-j-1)(d-i-1)$ nonintersecting edges.  However, $i(d-j-1)+j(d-i-1) \leq ij + (d-j-1)(d-i-1)$ when $0 \leq i \leq j \leq D $.  Therefore, the minimum number $M$ of pairs of nonintersecting edges over a drawing of the graph occurs when the $i$ and $j$ edges are arranged so that they lie in the same halfplane.  Thus, we assume that $0 \leq i \leq j \leq D $ and a given type $i,j$ edge always determines $i(d-j-1)+j(d-i-1)$ pairs of nonintersecting edges.  Summing this quantity over all edges of a drawing we obtain $$M=\displaystyle\sum_{i=0}^{D}\displaystyle\sum_{j=i}^{D} [i(d-j-1)+j(d-i-1)]x_{i,j}$$ pairs of nonintersecting edges.   \\ \\
In order to minimize $M$, we begin by multiplying equation ($1$) by $i(d-i-1)$ and subtracting it from $M$ for all values of $i$, yielding \begin{equation}  M= \displaystyle\sum_{i=1}^{D} i(d-i-1)y_i + \displaystyle\sum_{i=0}^{D-1} \displaystyle\sum_{j=i+1}^{D} (j-i)^2 x_{i,j}. \end{equation}

Let $p_{s,t}$ count the number of vertices having endvertices of type $s$ as the smallest type ($ 0 \leq s \leq D$).  The index $t$ counts the number of distinct sequences of endvertex types for a given vertex counted in $p_{s,t}$ ($t \geq 1$).  For example, in a convex drawing of $S_{n,d}$, $p_{0,1}=n$, $p_{0,t}=0$ for $t \geq 2 $, and $p_{s,t}=0$ for $s \geq 1 $.  Then, \begin{equation} n= \displaystyle\sum_{s=0}^{D} \displaystyle\sum_{t \geq 1} p_{s,t}.  \end{equation}  

Note that if the smallest type $s$ of an endvertex is $0$ then the point must be on the convex hull and all such points will have one distinct sequence of endvertex types.  Thus, $p_{0,t}=0$ for $t \geq 2 $.  \\
Let $z_{s,t,i}$ denote the number of endvertices of type $i$ for the $p_{s,t}$ vertices.  It follows that 
\begin{equation} 
 y_i=2p_{0,1} + \displaystyle\sum_{t \geq 1} \displaystyle\sum_{s=1}^{i} z_{s,t,i}p_{s,t} 
\end{equation}   \\
and for odd $d$ we have $$ y_D=p_{0,1} + \displaystyle\sum_{t \geq 1} \displaystyle\sum_{s=1}^{D} z_{s,t,D}p_{s,t}. $$

Additionally, since every vertex has $d$ edges, for a fixed $s$ and $t$ it holds that \begin{equation}   \displaystyle\sum_{i=s}^{D} z_{s,t,i}=d.  \end{equation} \\

Using equations ($3$) and ($4$) we obtain 
\begin{equation} y_i= 2n +  \displaystyle\sum_{t \geq 1}[ \displaystyle\sum_{s=1}^{i} (z_{s,t,i}-2)p_{s,t} - 2\displaystyle\sum_{s=i+1}^{D}p_{s,t}] 
\end{equation}  \\
and respectively, for $d$ odd we have 
$$y_D=n+ \displaystyle\sum_{t \geq 1} \displaystyle\sum_{s=1}^{D} (z_{s,t,i}-1)p_{s,t}.$$ \\
We proceed for $d$ even.  Using equation ($6$) we can rewrite the first part of the expression for $M$ in equation ($2$)  as 
$$\displaystyle\sum_{i=1}^{D} i(d-i-1)y_i=2n \displaystyle\sum_{i=1}^{D} i(d-i-1) + \displaystyle\sum_{t \geq 1} \displaystyle\sum_{i=1}^{D} i(d-i-1)[\displaystyle\sum_{s=1}^{i} (z_{s,t,i}-2)p_{s,t} - 2\displaystyle\sum_{s=i+1}^{D}p_{s,t}].$$
Following a change in the indices of the sums, the right term can be rewritten as $$2n \displaystyle\sum_{i=1}^{D} i(d-i-1) + \displaystyle\sum_{t \geq 1} \displaystyle\sum_{s=1}^{D} p_{s,t}[\displaystyle\sum_{i=s}^{D} i(d-i-1)(z_{s,t,i}-2)-2\displaystyle\sum_{i=1}^{s-1} i(d-i-1)].$$    
This can again be rewritten as \\
$$ 2n \displaystyle\sum_{i=1}^{D} i(d-i-1)+\displaystyle\sum_{t \geq 1} \displaystyle\sum_{s=1}^{D} p_{s,t}[s(d-s-1)\displaystyle\sum_{i=s}^{D} (z_{s,t,i}-2)+ $$ \\  $$\displaystyle\sum_{i=s+1}^{D}(i(d-i-1)-s(d-s-1))(z_{s,t,i}-2)-2\displaystyle\sum_{i=1}^{s-1} i(d-i-1)].$$

Using equation ($5$), it follows that this term is also equal to 
$$2n \displaystyle\sum_{i=1}^{D} i(d-i-1) + \displaystyle\sum_{t \geq 1} \displaystyle\sum_{s=1}^{D} p_{s,t}[C(s,d)+ \displaystyle\sum_{i=s+1}^{D}(i(d-i-1)-s(d-s-1))(z_{s,t,i}-2)]$$

where 
\begin{eqnarray*}
C(s,d)&=& s(d-s-1)(d - \displaystyle\sum_{i=s}^{D} 2) - 2\displaystyle\sum_{i=1}^{s-1} i(d-i-1) \\ 
&=& s(d-s-1)(d-2(D-s+1))- 2\displaystyle\sum_{i=1}^{s-1} i(d-i-1).
\end{eqnarray*}

We now show that $C(s,d)$ is nonnegative for all $s$ and $d$.  First, we have
$$s(d-s-1)(d-2(D-s+1) \geq s(d-s-1)(2s-1).$$  Then 

\begin{eqnarray*}
2\displaystyle\sum_{i=1}^{s-1} i(d-i-1) &\leq& 2\displaystyle\sum_{i=1}^{s-1}(s-1)(d-s) \\
&=& 2(s-1)^2(d-s) \\
&<& s(d-s-1)(2s-2).  
\end{eqnarray*}

Therefore $$C(s,d)>s(d-s-1)(2s-1)-s(d-s-1)(2s-2)=s(d-s-1) \geq 0.$$ Additionally, $(i(d-i-1)-s(d-s-1)) \geq 0 $ for $i,s \leq D=\frac{1}{2}(d-1)$ and $i \geq s+1$.  
Assuming $ z_{s,t,i} -2 \geq 0 $ (which we will prove in the following lemma) then the first half of the expression for $M$ is minimized when $p_{s,t}=0$ for all $s \geq 1$.  \\
Also, accounting for the discrepancy in $y_D$ when $d$ is odd, an analogous summation can be carried out.  Since the term $ z_{s,t,D} -1$ must be carried throughout this summation the expression for $d$ odd is also minimized for $p_{s,t}=0$ for all $s \geq 1$, provided $ z_{s,t,D} -1 \geq 0 $.

\begin{lem} $z_{s,t,i} \geq 2$ for all $s,t,i$, and $z_{s,t,D} \geq 1$ for $d$ odd.
\end{lem}

\begin{proof}
For a given vertex, we begin by proving there is at least one endvertex of type $\frac{1}{2}(d-1)$ for $d$ odd and there are at least two endvertices of type $\frac{1}{2}(d-2)$ for $d$ even.  This statement can be proved by induction from $d$ to $d+1$.  This statement is obvious for $d=2$ and $d=3$, so we begin with the inductive step.  Also, note that in traversing the $d$ edges  incident to a given vertex in a clockwise or counterclockwise manner in moving from edge to edge, edge to extension, extension to edge, and extension to extension, the number of edges in the clockwise following halfplane may change by at most one.  This fact will be used numerous times throughout the proof.   \\
\textbf{Case I:} From odd $d$ to $d+1$. \\
We consider the edge whose endvertex is of type $\frac{1}{2}(d-1)$ in the $d$-regular drawing. When the ($d+1$)st edge is added, this original endvertex will be the first endvertex of type $\frac{1}{2}[(d+1)-2]$.  If the ($d+1$)st edge is added in this edge's clockwise following halfplane then an immediately following edge or edge extension's endvertex will have type $\frac{1}{2}[(d+1)-2]$.  Thus, either this edge or the edge corresponding to this extension's endvertex will be the second endvertex of type $\frac{1}{2}(d-1)$. \\
\textbf{Case II:} From even $d$ to $d+1$. \\  
Consider an edge whose endvertex is of type $\frac{1}{2}(d-2)$ which has $\frac{1}{2}d$ edges in one of its halfplanes and $\frac{1}{2}(d-2)$ in the other.  If the ($d+1$)st edge is added in the halfplane with $\frac{1}{2}(d-2)$ edges then the considered endvertex is of type $\frac{1}{2}[(d+1)-1]$.  If the ($d+1$)st edge is added in the halfplane with $\frac{1}{2}d$ edges then there are $\frac{1}{2}[(d+1)+1]$ edges in this halfplane and $\frac{1}{2}[(d+1)-3]$ edges in the clockwise following halfplane of this edge's extension.  Since the number of edges in the clockwise following halfplane can change by at most one when moving from edge line to edge line (edge ray and edge extension), we find that traversing the graph from the edge with $\frac{1}{2}[(d+1)+1]$ edges in the clockwise following halfplane to the extension with $\frac{1}{2}[(d+1)-3]$ there must occur an edge or extension with $\frac{1}{2}[(d+1)-1]$ edges in the clockwise following halfplane.  Thus, this edge or the edge corresponding to the extension's endvertex is of type $\frac{1}{2}[(d+1)-1]$.  \\ \\
Using this result and the fact that in moving from edge line to adjacent edge line, the number of edges in the clockwise following halfplane may change by at most one, we can prove that there are two endvertices of each type from the minimal type $s$ to the maximal type $D$.  For $d$ odd, we have one endvertex of maximal type $D=\frac{1}{2}(d-1)$.  Traversing the $d$ edges starting and ending with the edge of type $D$ from edge line to edge line we must go down to an edge or an extension with $s$ edges in the clockwise following halfplane, and then back up to one with $D$.  Thus, we find there are at least two of edges or extensions whose endvertices are of each type from $s$ to $D$.  For $d$ even, we have two edges of maximal type $D=\frac{1}{2}(d-2)$.  Traversing the $d$ edges from one of the type $D$ edges to the other must go down to an edge or extension with $s$ edges in the clockwise following halfplane and back up to one with $D$.  Thus, there are at least two edges or extensions whose endvertices are of each type from $s$ to $D$.  It follows that $z_{s,t,i} \geq 2$.
\end{proof}
Going back to the final expression for equation ($2$) we have

$$M=2n \displaystyle\sum_{i=1}^{D} i(d-i-1)  + \displaystyle\sum_{t \geq 1} \displaystyle\sum_{s=1}^{D} p_{s,t} [C(s,d)+ \displaystyle\sum_{i=s+1}^{D}(i(d-i-1)-s(d-s-1))(z_{s,t,i}-2)] $$

\begin{center}
$$ + \displaystyle\sum_{i=0}^{D-1} \displaystyle\sum_{j=i+1}^{D} (j-i)^2 x_{i,j}. $$ 
\end{center}

Since $C(s,d)$, $z_{s,t,i} -2 $, and $(j-i)^2$ are greater than or equal to $0$ we find that this expression is minimized when $p_{s,t}=0$ for $s \geq 1$ and $x_{i,j}=0$ for $i<j$.  Evaluating the initial sum using these conditions we find that for even $d$ we have $$M=2n \displaystyle\sum_{i=1}^{D} i(d-i-1)=\frac{1}{6}nd(d-1)(d-2)$$ pairs of nonintersecting edges, and for odd $d$ we have $$M=2n \displaystyle\sum_{i=1}^{D-1} i(d-i-1) +nD(d-D-1)= \frac{1}{6}nd(d-1)(d-2)$$ pairs of nonintersecting edges. Since every pair of nonintersecting edges can count twice for two intersecting edges (see Figure 6) we can subtract at least $\frac{1}{12}nd(d-1)(d-2)$ from the initial upper bound  to obtain the asserted bound.

\begin{figure}[h]
\begin{center}
\includegraphics[scale=.33]{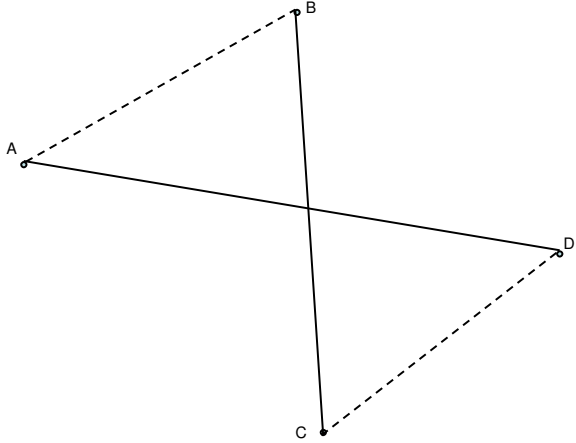}
\caption{Edges $AB$ and $CD$ are a pair of nonintersecting edges determined by both edge $BC$ and edge $AD$}
\end{center}
\end{figure}

\end{proof}

\subsection{Conjecture on the upper bound of $\overline{\rm CR}(R_{n,d})$}
For $n \equiv d \equiv 0 \pmod 2$ we have the following conjecture.
\begin{con}

The bound in Proposition 2.2 is sharp.
\end{con}

We offer a an alternate proof of $ \overline{CR}(R_{n,2})$, originally proven by Furry and Kleitman \cite{cycle1}, as a partial result in the direction of this conjecture.

\begin{prop}
$$ \overline{CR}(R_{n,2}) = \lfloor \frac{1}{4}n(2n-7) \rfloor \ \rm{where}\  n\equiv 0 \pmod 2. $$

\end{prop}

\begin{proof}

The lower bound follows from Proposition 2.2. Therefore, we proceed by proving the upper bound.  For each edge there is a maximum of $n-3$ nonadjacent edges which it can intersect.  Since $n \equiv 0 \pmod 2$, those edges which have $n-3$ crossings must have neighbor edges in different halfplanes. The two neighbor edges cannot have $n-3$ crossings since these edges cannot intersect each other. Thus there are at most $\frac{1}{2}n$ disjoint edges which may have $n-3$ crossings. It follows that $$ \overline{CR}(R_{n,2}) \leq \frac{1}{2} [ \frac{1}{2}n (n-3) + \frac{1}{2}n(n-4)]= \lfloor \frac{1}{4}n(2n-7) \rfloor. $$

\end{proof}
Note that only for $d=2$ and $n$ even there occur disconnected graphs $S_{n,2}$ in the extremal cases, that is, there are copies of $C_4$ if $n \equiv 0\pmod 4$ and there are copies of $C_4$ and one copy of $C_6$ if $n \equiv 2\pmod 4$.

\subsection{Alternate proof of $\overline{CR}(C_n)$}

In the same vein as the above proof for $ \overline{CR}(R_{n,2})$ we now offer a simpler proof of $\overline{CR}(C_n)$ than that of Furry and Kleitman.  Note that for both $R_{n,2}$ and $C_n$ where $n\equiv 1 \pmod 2$ the proof of the maximum rectilinear crossing number is trivial as both achieve the thrackle bound.

\begin{prop}  
$$ \overline{CR}(C_n) = \frac{1}{2}(n^2-4n+2) \rm{where} n \equiv 0 \pmod 2.$$
\end{prop}

\begin{proof}
The lower bound follows from \cite{cycle1}. Therefore, we proceed by proving the upper bound.  In an even cycle an edge with $n-3$ crossings must have its neighbor edges in different halfplanes.  Assume that we have three such edges.  We label these three pairwise intersecting edges $A_1A_2$, $B_1B_2$, and $C_1C_2$ as shown in Figure $7$.

\begin{figure}[h]
\begin{center}
\includegraphics[scale=.45]{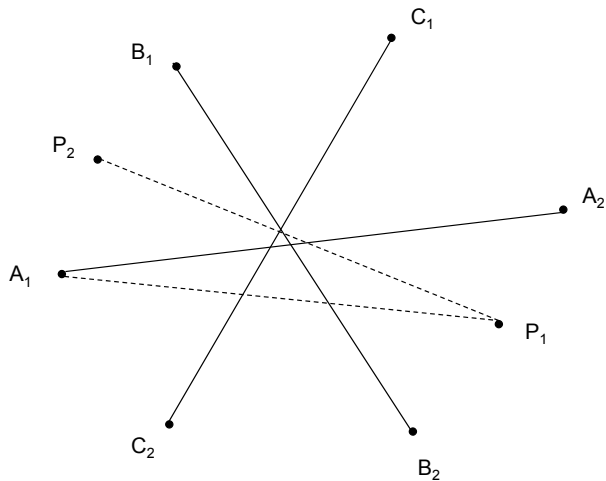}
\caption{The three edges $A_1A_2$, $B_1B_2$, and $C_1C_2$ are assumed to have $n-3$ crossings each.}
\end{center}
\end{figure}

We start at $A_1$.  The other edge incident to $A_1$ must intersect both  $B_1B_2$ and $C_1C_2$.  Thus, without a loss of generality, we may assume that the termination of this edge, $P_1$, must lie in region from $A_2$ to $B_2$.  The next edge must terminate at $P_2$ which must lie in the region from $A_1$ to $B_1$, and $P_3$ must lie again in the region from $A_2$ to $B_2$, and so on, since all three original edges must be intersected by all edges except their neighbor edges.  Eventually, the cycle must close up on itself and $P_{2i-1}$ or $P_{2i}$ must terminate at $B_1$ or $B_2$, respectively, since edges $A_1A_2$ and $C_1C_2$ must be intersected.  Note that the cycle cannot close on $A_2$ or any $P_k$ where $k < i$, because this will result in a disconnected two-regular drawing.  It follows that edge $B_2P_{2i+1}$ or $B_1P_{2i+2}$ must have $P_{2i+1}$ or $P_{2i+2}$ in the region from $A_1$ to $B_1$ or $A_2$ to $B_2$, respectively.  Thus, $C_1$ or $C_2$ can never be reached without forcing one of the edges $A_1A_2$ and $B_1B_2$ to have less than $n-3$ crossings, a contradiction.  It follows that at most two edges can have $n-3$ crossings and thus we obtain $$ \overline{CR}(C_n) \leq \frac{1}{2} [2(n-3)+(n-2)(n-4)]=\frac{1}{2}(n^2-4n+2).$$
 \end{proof}
 
Another partial result in the direction of Conjecture 3.3 is the following proposition.
 \begin{prop}
 $$ \overline{CR}(R_{n,n-2})=\binom{n}{4}$$ for $n$ even.
 \end{prop}
 
\begin{proof}
  
The lower bound follows from Proposition 2.2 where every second
side is deleted from a rectilinear drawing of $K_n$ as a convex $n$-gon.
This bound is sharp since every $4$-tuple of vertices can determine at most
one crossing.
\end{proof}

\subsection{A generalization of previous results}

Theorem 3.1 extends known results regarding the maximum rectilinear crossing number of the cycle and complete graph to the more general class $R_{n,d}$ of $d$-regular graphs where $2\leq d \leq n-1$.  We remark here that when we substitute $d=2$ into Theorem 3.1 we have $$\overline{\rm CR}(R_{n,2})=\overline{\rm CR}(C_n)=\frac{1}{4}(2n)(3(2)n - 2(2)^2-6(2)+2)=\frac{1}{2}n(n-3).$$  This is the same result obtained in \cite{cycle1} for $\overline{\rm CR}(C_n)$ where $n \equiv 1 \pmod 2$.  \\
Additionally, we can substitute $d=n-1$ into Theorem 3.1 yielding 
\begin{eqnarray*}
\overline{\rm CR}(R_{n,n-1})&=&\overline{\rm
CR}(K_n)=\\
&=&\frac{1}{24}n(n-1)[3n(n-1)-2(n-1)^2-6(n-1)+2] =\\
&=&\frac{1}{24}n(n-1)(n-2)(n-3)=\binom{n}{4}.
\end{eqnarray*} \\
This is the same result obtained in \cite{complete1} regarding $\overline{\rm
CR}(K_n)$.

\subsection{Computational results}

The following table shows the values of $\overline{\rm
CR}(R_{n,d})$ for various $n$ and $d$.  Note that the values in bold are the conjectured results.

$$
\begin{tabular}[t]{|l|l|l|l|l|l|l|l|}
\hline
$d \backslash n$ & 4&5&6&7&8&9&10 \\
\hline
2&-&5&7&14&18&27&32\\
\hline
3&1&-&15&-&38&-&70 \\
\hline
4&-&5&15&35&\textbf{52}&81&\textbf{105} \\
\hline
5&-&-&15&-&70&-&150 \\
\hline
6&-&-&-&35&70&126&\textbf{133}\\
\hline
7&-&-&-&-&70&-&210 \\
\hline
8&-&-&-&-&-&126&210 \\
\hline
9&-&-&-&-&-&-&210 \\
\hline
\end{tabular}
$$
\subsection{A general conjecture}
For the determination of the maximum rectilinear crossing number of any graph $G$ it would be very helpful if the following conjecture can be proved.

\begin{con}
The maximum rectilinear crossing number of any graph can be realized in a drawing where all the vertices are vertexpoints of a convex polygon.
\end{con}

\section*{Acknowledgments}

The authors would like to thank David Garber for fruitful discussions.

\end{document}